\documentclass{amsart}

\usepackage{amsmath}
\usepackage{amsthm}
\RequirePackage[usenames,dvipsnames]{color}
\usepackage{mathrsfs}
\usepackage{amssymb}
\usepackage[hidelinks]{hyperref}
\usepackage{stmaryrd}
\usepackage{graphicx}
\usepackage{enumitem} 
\usepackage{txfonts} 
\usepackage{comment}
\usepackage[normalem]{ulem}

\usepackage[all,pdftex, cmtip]{xy}
\newdir{ >}{{}*!/-10pt/@{>}}

\usepackage{tikz-cd}
\usetikzlibrary{decorations.pathmorphing}

\usepackage{todonotes}

\DeclareMathAlphabet{\mathbbe}{U}{bbold}{m}{n}

\newcommand{\kto}{\rightsquigarrow}

\newtheorem{thm}{Theorem}[section]
\newtheorem{lem}[thm]{Lemma}
\newtheorem{prop}[thm]{Proposition}

\newtheorem{obs}[thm]{Observation}

\theoremstyle{definition}
\newtheorem{defn}[thm]{Definition}
\newtheorem{ex}[thm]{Example}

\theoremstyle{remark}
\newtheorem*{rmk}{Remark}

\makeatletter
\@addtoreset{section}{part}
\makeatother  

\newcommand{\RR}{\mathbb{R}}

\newcommand{\cA}{\mathcal{A}}
\newcommand{\cB}{\mathcal{B}}
\newcommand{\cC}{\mathcal{C}}

\newcommand{\DDelta}{\mathbbe{\Delta}}

\newcommand{\CR}{\mathrm{cr}}

\newcommand{\defeq}{\mathrel{:=}}

\ifx\color@rgb\@undefined\else
\definecolor{violet}{rgb}{0.7,0,1}
\fi

\newcommand{\hey}[1]{\color{magenta}#1\color{black}~}

\begin{document}

	\title{A first step toward higher order chain rules in abelian functor calculus}
	\author[Osborne]{Christina Osborne}\thanks{C. Osborne, University of Virginia, cdo5bv@virginia.edu, 434-243-1115}
	\author[Tebbe]{Amelia Tebbe}\thanks{A. Tebbe, Indiana University Kokomo, antebbe@gmail.com}


	
	\date{\today}
	
	\begin{abstract}
		
		One of the fundamental tools of undergraduate calculus is the chain rule. 
		The notion of higher order directional derivatives was developed by Huang, Marcantognini, and Young, along with a corresponding higher order chain rule. When Johnson and McCarthy established abelian functor calculus, they proved a chain rule for functors that is analogous to the directional derivative chain rule when $n = 1$. In joint work with Bauer, Johnson, and Riehl, we defined an analogue of the iterated directional derivative and provided an inductive proof of the analogue to the chain rule of Huang et al.\\
		This paper consists of the initial investigation of the chain rule found in Bauer et al., which involves a concrete computation of the case when $n=2$.  We describe how to obtain the second higher order directional derivative chain rule for abelian functors.   This proof is fundamentally different in spirit from the proof given in Bauer et al. as it relies only on properties of cross effects and the linearization of functors.
	\end{abstract}

	\maketitle
	
	\setcounter{tocdepth}{1}
	\tableofcontents

	\section{Introduction}

	In this paper, we consider {\it abelian functor calculus}, the calculus of functors of abelian categories established by Brenda Johnson and Randy McCarthy (see \cite{JM:Deriving}).  Functor calculus enjoys certain properties that are analogous to results in undergraduate calculus.  This paper is a companion to \cite{BJORT}, in which many of these analogies are made explicit.  In \cite{BJORT}, one of the main results \cite[Theorem 8.1]{BJORT} provides a chain rule for the $n$th higher order directional derivative, denoted as $\Delta_n$ \cite[Definition 7.3]{BJORT},  associated to a functor between abelian categories.
	
	In order to arrive at this chain rule, we started with an explicit calculation in the case when $n=2$. It became clear that this method of calculation would not lend itself well to an inductive proof for a general $n$, which is why the proof of \cite[Theorem 8.1]{BJORT} is different in spirit from the approach in this paper. Because this result is quite technical and lengthy, it warranted independent documentation.  However, the groundwork - including the definitions and properties of most of the functors we will use - is already documented in \cite{BJORT}.  For this reason, we will heavily cite \cite{BJORT} throughout this paper.

	The goal of this paper is a chain rule for the second order directional derivative of a functor $F$, which is stated in the following theorem.
	{	\renewcommand{\thethm}{\ref{DirectionalChainRule}}
		\begin{thm} 
			Given two composable functors $G:\cA \rightarrow \cB$ and $F:\cB \rightarrow \cC$ with object $x$, $v$, and $w$ in $\cA$,  there is a chain homotopy equivalence  
			\begin{equation*}\Delta_2(F\circ G)(w,v;x)\simeq \Delta_2F(\Delta_2G(w,v;x),\Delta_1G(v;x);G(x)).\end{equation*}
		\end{thm}
		\addtocounter{thm}{-1}
		
		
		The directional derivatives are defined in Section \ref{DirectionalSection}. The left and right sides of this equivalence are written in terms of the smallest component parts of the functors in Sections 3 and 4, respectively. These smallest components are the cross effects of $F$ and $G$;
		cross effects of functors are defined in Section \ref{CrSection}.    
		The proof of Theorem \ref{DirectionalChainRule} is concluded in Section \ref{proof} by matching terms (\ref{A1}) through (\ref{A31}) from Section \ref{LHS} with homotopy equivalent terms (\ref{B1}) through (\ref{B20}) from Section \ref{RHS}.  
		
	}
	\subsection{Acknowledgments}
	
	The authors would like to thank the Banff International Research Station (host of the second Women in Topology workshop) and the Pacific Institute for Mathematical Sciences for providing us with the opportunity to collaborate along with the other authors of \cite{BJORT}. We would also like to express our gratitude to Kristine Bauer, Brenda Johnson, and Emily Riehl for their guidance and support throughout this project.
	
	
	
	\section{Cross effects, linearization, and directional derivatives}
	
	In this section we provide the foundational tools and motivation for the main result, which is the second higher order directional derivative chain rule (Theorem \ref{DirectionalChainRule}). The construction of higher directional derivatives is possible once cross effects and linearizations of functors and some of their key properties are obtained. Thus, we begin by defining the cross effects and linearizations of functors.
	
	\subsection{Cross effects and linearization}\label{CrSection}

	The following definition appears as \cite[Definition 2.1]{BJORT}.
	
	\begin{defn}[\cite{EM}]\label{def:CR}  The \emph{$n$th cross effect} of a functor $F \colon \cA \to \cB$ between two abelian categories, where the zero object of $\cA$ is denoted by $0$, is the functor \[\CR_nF\colon \cA^{n} \to \cB\] defined recursively by 
		\[ F(x)\cong F(0)\oplus \CR_1F(x)\]
		\[ \CR_1F(x_1\oplus x_2) \cong \CR_1F(x_1)\oplus \CR_1 F(x_2) \oplus \CR_2 F(x_1, x_2)\]
		and in general,
		\begin{align*} \CR_{n-1}F(x_1\oplus x_2, x_3,  \ldots , x_n)\cong & \CR_{n-1}F(x_1, x_3, \ldots , x_n) \oplus \CR_{n-1}F(x_2, x_3, \ldots , x_n) \\ &\oplus \CR_nF(x_1, x_2, \ldots, x_n),\end{align*}
		where $\oplus$ denotes the biproduct in both categories $\cA$ and $\cB$ (a common abuse of notation).
	\end{defn}
	
	In this paper, we will call any functor between abelian categories an \emph{abelian functor}.

	The cross effect functor is \emph{multi-reduced} in the following sense.
	
	\begin{prop}\cite[Proposition 1.2]{JM:Deriving}\label{multi-reduced}
		For an abelian functor $F:\cA\to\cB$ and objects $x_1,\ldots,x_n$ in $\cA$, if any $x_i=0$, then 
		\[\CR_n F(x_1,\ldots, x_n)\cong 0.\]
	\end{prop}

	Let $C_nF$ denote the $n$th cross effect of $F$, $\CR_n$, composed with the diagonal functor.  That is,  $C_nF(x) \defeq \CR_nF(x, \ldots , x)$.  Corollary 2.7 of \cite{BJORT} shows that $\CR_n$ and the diagonal functor are an adjoint pair of functors, so that $C_n$ is a comonad.  Let $\epsilon$ denote the counit of this comonad.  
	
	Functor calculus studies approximations of functors that behave like degree $n$ polynomials, up to chain homotopy equivalence.  Let $\simeq$ denote chain homotopy equivalence. The following definition makes the idea of polynomial degree $n$ functors precise.
	
	\begin{defn}
		An abelian functor $F:\cA\rightarrow\cB$ is \emph{degree n} if
		\[\CR_{n+1} F\simeq 0.\]
		In particular, $F$ is degree 1 if $\CR_2 F\simeq 0$.  If $F$ is also reduced, meaning that $F(0)\simeq 0$, then we say that $F$ is \emph{linear}. We call $F$ \emph{strictly reduced} if $F(0)\simeq 0$ is a chain homotopy equivalence. 
	\end{defn}
	
	\begin{defn} \label{def:D1} \cite[Definition 5.1]{BJORT} The \emph{ linearization} of an abelian functor $F\colon \cA \to \cB$ is the (abelian) functor $D_1F\colon \cA\to { Ch} \cB$ given as the explicit chain complex $(D_1F_*, \partial_*)$ where:
		\[ (D_1F)_k \coloneqq \begin{cases} 
		C_2^{\times k}F& k\geq 1\\
		\CR_1F & k=0 \\
		0 & \text{otherwise}\end{cases}\]
		and $C_2^{\times k}$ is the functor $C_2= \CR_2\circ diag$ composed with itself $k$ times.

		The chain differential $\partial_1\colon (D_1F)_1 \to (D_1F)_0$ is given by the map $\epsilon$ of \cite[Remark 2.8]{BJORT}, and the chain differential $\partial_k \colon (D_1F)_k \to (D_1F)_{k-1}$ is given by $\sum_{i=1}^k (-1)^i C_2^{\times i}\epsilon$ when $k\geq 1$, where $\epsilon$ is the counit of the adjunction of \cite[Corollary 2.7]{BJORT}.
		
	\end{defn}
	Note that, as just defined, $D_1F$  is a functor whose codomain is the category of {\it chain complexes} in $\cB$.  This causes a potential problem when composing functors. Intuitively, we would expect a relationship between $D_1(F\circ G)$ and $D_1F\circ D_1G$ (this will be made explicit in Lemmas \ref{reducedChainRule} and \ref{D1ChainRule}).  However, a priori, if $G:\cA\to \cB$ and $F:\cB\to \cC$ then $D_1F:\cB\to Ch\cC$ and $D_1G:\cA\to Ch \cB$ are not composable.  To remedy this, we define this composite using the Dold-Kan correspondence (see, e.g. \cite[8.4]{Weibel}).  Define $D_1F\circ D_1G$ by 
	
	\[
	\begin{tikzcd} 
	\cA \ar[r, "D_1G"] & Ch\cB \ar[r, "K", "\simeq"'] & \cB^{\DDelta^{op}} \ar[r, "D_1F^{\DDelta^{op}}"] & Ch \cB^{\DDelta^{op}} \ar[r, "N", "\simeq"'] & Ch Ch \cB \ar[r, "Tot"] & Ch\cB
	\end{tikzcd}
	\]
	where $K$ and $N$ denote the inverse functors of the Dold-Kan equivalence, and $Tot$ is the totalization functor.  When we compose functors such as these, we implicitly mean that we use this procedure to compose.  This process is made effortless by using the structure of a Kleisli category, which is explained in detail in Section 3 of \cite{BJORT}.  For the purposes of this paper, we will suppress the $Ch$ in the codomain of functors, and write $D_1G:\cA \kto  \cB$ instead of $D_1G:\cA \to Ch\cB$ to indicate that we are working in this Kleisli category and to avoid cumbersome $Ch$'s.

	The linearization functor satisfies the following properties:
	
	\begin{lem} \label{lem:D1-linearity} \cite[Lemma 5.6]{BJORT}  $\quad$
		\begin{enumerate}
			\item For any abelian functor $F \colon \cA \to \cB$, the functor $D_1F \colon \cA \kto { \cB}$ is strictly reduced,  and for any $x, y \in \cA$, the natural map
			\[ D_1F(x) \oplus D_1F(y) \to D_1F(x \oplus y)\] is a chain homotopy equivalence.  In particular, $D_1F$ is linear.
			\item  
			In the category of chain complexes of abelian categories, pointwise chain homotopy equivalence classes in $Fun(\cA,\cB)$ are denoted by $[\cA,\cB]$. The functor $D_1 \colon [\cA,\cB] \to [\cA,\cB]$ is linear in the sense that $D_10 \cong 0$ and for any pair of functors $F, G \in [\cA,\cB]$, \[D_1F \oplus D_1G \cong D_1(F \oplus G).\]
		\end{enumerate}
	\end{lem}
	
	We will follow Convention 5.11 in \cite{BJORT}. In particular, given $F:\cA^n\rightarrow\cB$, consider $F_i:\cA\rightarrow\cB$ defined by 
	\[F_i(y):=F(x_1,\ldots,x_{i-1},y,x_{i+1},\ldots, x_n),\]
	where $x_1,\ldots,x_{i-1},x_{i+1},\ldots,x_n$ are fixed objects of $\cA$. We will write $D_1^iF(x_1,\ldots,x_n)$ for $D_1 F_i(x_i)$. 
	In cases where a single variable $x_i$ occurs in multiple inputs of a multi-variable functor $F$, and we wish to indicate simultaneous multi-linearization of all occurrences of $x_i$, we will use the notation $D_1^{x_i}F$. 
	Let us look at a specific example to see how $D_1^{x_i}F$ works.
	\begin{ex}
		Let $F:\cA^4 \to \cB$ and consider $D_1^xF(x,y,x,z)$.  Define $G:\cA\to\cB$ as 
		\begin{align*}
		G(x)&\defeq F(x,y,x,z)\\
		&= F(-,y,-,z)\circ diag(x)
		\end{align*}
		where $diag$ is the diagonal functor. Then $D_1^xF(x,y,x,z)\defeq D_1G(x)$.
	\end{ex}
	
	We will frequently use these two notations when dealing with cross effects. 
	For example, in the proof of Lemma \ref{nabla2} we consider the sequential linearizations $D_1^1 D_1^2 \CR_2 F(a,b)$ of the two-variable functor $\CR_2 F$.  In this case, we linearize this functor in each variable separately. On the other hand, in Lemma \ref{BJORTCor5.13}, we consider functors such as $D_1^x \CR_2 F(x, x)$, which is the linearization of the functor $\CR_2F$ in its two variables simultaneously. These two linearization processes produce quite different results.

	\subsection{Higher order directional derivatives}	\label{DirectionalSection} \label{Section:FunctorDirectionalDerivatives}
	
	

	Recall that the directional derivative of a differentiable function $f \colon \RR^n \to \RR^m$ at the point $x\in \RR^n$ in the direction $v\in \RR^n$ measures how the value of $f$ at $x$ changes while translating along the infinitesimal vector from $x$ in the direction $v$. One way to make this idea precise is to define $\nabla f(v;x)$ to be the derivative of the composite function, substituting the affine linear function $t \mapsto x + tv$ into the argument of $f$, evaluated at $t=0$:
	\[\left.\Delta_1f(v;x)\defeq\nabla f(v;x) = \frac{\partial}{\partial t}f \left(x+tv\right)\right|_{t=0}.\]
	In \cite{HMY}, it was shown that the first directional derivative has a chain rule:
	\[\Delta_1 (f\circ g)(v;x)=\Delta_1f(\Delta_1g(v;x);g(x)).\]
	Using the first directional derivative, we can define the second directional derivative:
	\[ \left.\Delta_2f(w,v;x)\defeq \frac{\partial}{\partial t}\Delta_1f\left(v+tw;x+tv\right)\right|_{t=0},\]
	which also has a chain rule \cite{HMY}:
	\[\Delta_2(f\circ g)(w,v;x)=\Delta_2f(\Delta_2g(w,v;x),\Delta_1g(v;x);g(x)).\]
	Generally speaking, for any $n$ there is a higher order directional derivative along with a corresponding chain rule (see \cite[Theorem 3]{HMY}).	
	
	When Johnson and McCarthy established abelian functor calculus, they constructed an analog to the first directional derivative along with a chain rule \cite[Proposition 5.6]{JM:Deriving}. The formula for this chain rule from \cite{JM:Deriving} mirrors the case when $n=1$ in the directional derivative chain rule for functions found in \cite{HMY}.  These similarities provide the motivation to pursue higher order directional derivatives of abelian functors in the hopes of acquiring an analogous higher order chain rule. 
	
	
	
	An equivalent notion of the directional derivative of a functor was developed in \cite{BJORT} and will be used here.

	\begin{defn}\label{defn:dir-derivative} \cite[Lemma 6.3]{BJORT}
		For an abelian functor $F \colon \cA\to\cB$ with objects $x$ and $v$ in $\cA$, define
		\[ \nabla F(v;x) \defeq D_1 F(x \oplus -)(v).\]
	\end{defn}

	When defining higher order directional derivatives for functors in \cite{BJORT}, the goal was to imitate the iterative process used to define the higher directional derivatives of functions in \cite{HMY}.
	
	\begin{defn} \label{def:delta_2} \cite[Definition 7.3]{BJORT}  Consider an abelian functor $F:\cA\rightarrow \cB$ and objects $x$, $v$, and $w$ in $\cA$. The   \emph{higher order directional derivatives} of $F$ are defined recursively by 
		\begin{align*} \Delta_0F(x) &\defeq F(x),\\ 
		\Delta_1F(v; x) &\defeq \nabla F(v; x),\\
		\Delta_2 F(w,v; x) &\defeq \nabla(\Delta_1 F)\left( (w; v); (v; x) \right).\end{align*}
		We say that $\Delta_1F(v;x)$ is the \emph{first directional derivative} of $F$ at $x$ in the direction $v$. Similarly, we say that $\Delta_2F(w,v;x)$ is the \emph{second higher order directional derivative} of $F$ at $x$ in the directions $v$ and $w$.
	\end{defn}
	
	\begin{rmk}
		As in \cite{HMY}, we can continue and define the $n$th directional derivative, but in this paper we will stop at $n=2$. The full definition can be found in \cite[Definition 7.3]{BJORT}.
	\end{rmk}

	It was shown in \cite[Proposition 5.6]{JM:Deriving} that the  first directional derivative has a chain rule up to quasi-isomorphism.  The chain rule is now strengthened to a chain homotopy equivalence.
	
	\begin{thm}\cite[Theorem 6.5(v)]{BJORT}
		Given two composable abelian functors $G:\cA \rightarrow \cB$ and $F:\cB \rightarrow \cC$,  there is a chain homotopy equivalence
		\[\Delta_1(F\circ G)(v;x)\simeq \Delta_1F(\Delta_1G(v;x);G(x)).\]
	\end{thm}

	This brings us to the formulation of our main theorem.
	
	\begin{thm}\label{DirectionalChainRule} 
		Given two composable abelian functors $G:\cA \rightarrow \cB$ and $F:\cB \rightarrow \cC$ with object $x$, $v$, and $w$ in $\cA$,  there is a chain homotopy equivalence  
		\begin{equation*}\Delta_2(F\circ G)(w,v;x)\simeq \Delta_2F(\Delta_2G(w,v;x),\Delta_1G(v;x);G(x)).\end{equation*}
	\end{thm}

	We prove Theorem \ref{DirectionalChainRule} by expanding both sides and showing they are equivalent. The idea of the proof is to break each side down into direct sums of the smallest component parts.  These smallest component parts are linearizations  of compositions of cross effects for $F$ and $G$. The left-hand side is expanded in Section \ref{LHS}, the right-hand side is expanded in Section \ref{RHS}, and the terms from both sides are aligned to finish the proof in Section \ref{proof}.
	
	Before moving on to the expansions of each side, equivalent formulations of the first and second directional derivatives are stated. 
	
	We will use the following chain homotopy equivalent formulation of $\Delta_1$.
	
	\begin{lem}\label{nabla1}\cite[Lemma 6.3]{BJORT} For a functor $F:\cA\rightarrow \cB$ between abelian categories and any pair of objects $x,v\in \cA$, there is a chain homotopy equivalence
		$$\Delta_1F(w;x)\simeq D_1F(w)\oplus D_1^1\CR_2F(w,x).$$
	\end{lem}

	
	
	Similarly, there is an equivalent formulation for $\Delta_2$.
	
	\begin{lem}\label{nabla2}
		For a functor $F: \cA\rightarrow \cB$ between abelian categories, there is a chain homotopy equivalence, 
		\[\Delta_2F(w,v;x)\simeq D_1F(w)\oplus D_1^1\CR_2F(w,x)\oplus D_1^1D_1^2\CR_2F(v,v)\oplus D_1^1D_1^2\CR_3F(v,v,x).\]
	\end{lem}
	
	\begin{proof}
		By \cite[Lemma 6.9]{BJORT}, for objects $a$, $b$, $c$, and $d$ in $\cA$, we have that
		\[\nabla(\nabla F)((d;c);(b;a))\simeq D_1F(d)\oplus D_1^1\CR_2 F(d,a)\oplus D_1^1 D_1^2\CR_2 F(b,c)\oplus D_1^1 D_1^2\CR_3 F(b,c,a).\]
		Letting $a\defeq x$, $b\defeq v$, $c\defeq v$, and $d 
		\defeq w$, we get our desired result:
		\begin{align*}
		\Delta_2(F\circ G)(w,v;x)
		&=\nabla(\Delta_1F)((w,v);(v;x))\\
		&\simeq D_1F(w)\oplus D_1^1\CR_2F(w,x)\oplus D_1^1D_1^2\CR_2F(v,v)\oplus D_1^1D_1^2\CR_3F(v,v,x),
		\end{align*}
		where, for example, $D_1^1\CR_2F(w,x)\defeq D_1(\CR_2F(-,x))(w)$.
	\end{proof}

	%
	%
	%

	\section{The second order directional derivative of a composition}\label{LHS}

	In this section, we expand the left-hand side of Theorem \ref{DirectionalChainRule}. We begin by applying Lemma \ref{nabla2}.  We then use Lemma \ref{JMProp1.6} to rewrite the cross effects of the composition of functors in terms that are more manageable.  Finally, we use Lemma \ref{BJORTCor5.13}, which shows that a majority of the terms are contractible.
	
	Before explicitly computing $\Delta_2(F\circ G)$ in terms of cross effects, it will be useful to understand how we can rewrite terms such as $D_1^{x_2}\CR_2(F\circ G)(x_1,x_2)$. 
	Recall that \cite{JM:Deriving} provides a formula for the $p$th cross effects of a composition of functors.

	\begin{lem}\label{JMProp1.6} \cite[Proof of Proposition 1.6]{JM:Deriving}
		Let $G:\cA\rightarrow\cB$ and $F:\cB\rightarrow\cC$ be functors between abelian categories. Let $x_1,\ldots,x_p$ be objects in $\cA$ and $\langle p\rangle=\{1,2,\ldots,p\}$. For $U=\{s_1,\ldots,s_t\}\subseteq \langle p\rangle$, let $\CR_UG$ denote  $\CR_tG(x_{s_1},\ldots,x_{s_t})$ and for $U=\emptyset$, let $\CR_UG=G(0)$. Then
		$$\CR_p(F\circ G)(x_1,\ldots,x_p)\cong \bigoplus_{\{U_1,\ldots,U_k \vert  U_i\neq U_j\}\subseteq \mathcal{P}(\langle p \rangle),\cup_{i=1}^kU_i=\langle p \rangle} \CR_kF(\CR_{U_1}G,\ldots,\CR_{U_k}G).$$
	\end{lem}
	
	Let us see explicitly what this formula gives in a simple case.
	
	\begin{ex} \label{CompositionEx}
		Consider the case when $p=2$. Note that we can cover the set $\langle p \rangle$ with up to four distinct subsets since the cardinality of $\mathcal{P}(\langle p \rangle)$ is $4$.  To cover the set $\langle p \rangle$ with one set, it must be itself. To cover $\langle p \rangle$ with two subsets, there are four possibilities: $\{\{1,2\}\,\emptyset\}$, $\{\{1\},\{2\}\}$, $\{\{1\},\{1,2\}\}$, and $\{\{2\},\{1,2\}\}$. Similarly, there are four different ways to cover $\langle p \rangle$ with three subsets, and one way to cover $\langle p \rangle$ with four subsets. Applying the formula from Lemma \ref{JMProp1.6} gives 
		\begin{align*}
		\CR_2&(F\circ G)(x_1,x_2) \\
		&\cong \CR_1F(\CR_2G(x_1,x_2)) \oplus \CR_2F(\CR_2G(x_1,x_2),G(0)) \oplus \CR_2F(\CR_1G(x_1),\CR_1G(x_2)) \\
		&\quad \oplus \CR_2F(\CR_1G(x_1),\CR_2G(x_1,x_2))
		\oplus \CR_2F(\CR_1G(x_2),\CR_2G(x_1,x_2))  \\
		& \quad \oplus  \CR_3F(\CR_1G(x_1),\CR_1G(x_2),G(0)) \oplus \CR_3F(\CR_1G(x_1),\CR_2G(x_1,x_2),G(0)) \\
		& \quad \oplus \CR_3F(\CR_1G(x_2),\CR_2G(x_1,x_2),G(0)) \oplus \CR_3F(\CR_1G(x_1),\CR_1G(x_2),\CR_2G(x_1,x_2))\\
		& \quad \oplus \CR_4F(\CR_1G(x_1),\CR_1G(x_2),\CR_2G(x_1,x_2),G(0)).
		\end{align*}
	\end{ex}
	
	
	In order to simplify the expansion, we will use Lemma \ref{BJORTCor5.13} to conclude that some of the summands are in fact contractible.

	\begin{lem}\label{BJORTCor5.13} \cite[Corollary 5.13]{BJORT} Suppose the abelian functor $F:\cA\rightarrow \cB$ factors as 
		\[
		\begin{tikzcd} 
		\cA \arrow[rr, "F"]  \arrow[dr, "diag"'] & & \cB \\
		& \cA^n \arrow[ru, "H"'] &
		\end{tikzcd}
		\]
		where $diag:\cA\rightarrow \cA^n$ is the diagonal functor and $H$ is strictly multi-reduced.  Then $D_1F$ is contractible.
	\end{lem}
	
	Recall that cross effects are strictly multi-reduced functors. Note that the composition of strictly multi-reduced functors is still strictly multi-reduced.

	\begin{ex}\label{diag0ex}
		Using the expansion from Example \ref{CompositionEx}, let us compute $D_1^{x_2}\CR_2(F\circ G)(x_1,x_2)$. As noted previously, we can distribute $D_1^{x_2}$ to each summand. Since we are linearizing with respect to $x_2$ and since the functor $\CR_2F(\CR_1G(-),\CR_2G(x_1,-))$ is strictly multi-reduced, Lemma \ref{BJORTCor5.13} tells us that, for example,
		\begin{align*}
		D_1^{x_2}\CR_2F(\CR_1G(x_2),\CR_2G(x_1,x_2))
		&=D_1^{x_2}\CR_2F(\CR_1G(-),\CR_2G(x_1,-)) \circ diag(x_2) \\
		&\simeq 0,
		\end{align*}
		where $diag:\mathcal{A}\rightarrow \mathcal {A} \times \mathcal {A}$ is the diagonal functor $x_2\mapsto(x_2,x_2)$. Hence
		
		\begin{align*}
		D_1^{x_2}\CR_2&(F\circ G)(x_1,x_2) \\
		&\simeq D_1^{x_2} \CR_1F(\CR_2G(x_1,x_2))
		\oplus D_1^{x_2}\CR_2F(\CR_2G(x_1,x_2),G(0))\\
		&\quad \oplus D_1^{x_2}\CR_2F(\CR_1G(x_1),\CR_1G(x_2)) 
		\oplus D_1^{x_2}\CR_2F(\CR_1G(x_1),\CR_2G(x_1,x_2)) \\
		& \quad \oplus  D_1^{x_2}\CR_3F(\CR_1G(x_1),\CR_1G(x_2),G(0)) 
		\oplus D_1^{x_2}\CR_3F(\CR_1G(x_1),\CR_2G(x_1,x_2),G(0)).
		\end{align*}
	\end{ex}
	
	We proceed with the expansion of the left-hand side of \ref{DirectionalChainRule}. To start, applying Lemma \ref{nabla2} gives 
	\begin{align}
	\label{LHSFourSummands}\Delta_2(F\circ G)(w,v;x)\simeq & D_1(F\circ G)(w)\oplus D_1^2\CR_2(F\circ G)(x,w) \\
	&\nonumber \oplus D_1^1D_1^2\CR_2(F\circ G)(v,\bar{v})\oplus D_1^2D_1^3\CR_3(F\circ G)(x,v,\bar{v}),
	\end{align}
	where the new variable $\bar{v}:=v$ is introduced to better illustrate the computations. This expression provides the foundation for expanding the left-hand side. 
	We expand the second, third, and fourth terms of (\ref{LHSFourSummands}) using the methods illustrated in Examples \ref{CompositionEx} and \ref{diag0ex}. Specifically, Lemma \ref{JMProp1.6} is used to rewrite the cross effect, then Lemma \ref{lem:D1-linearity} is used to distribute the linerization functor(s) to each summand, and finally Lemma \ref{BJORTCor5.13} is applied to find the terms that are contractible.  The first term of (\ref{LHSFourSummands}) will be addressed later in Section \ref{proof}.
	
	
	\newpage
	For the second summand of (\ref{LHSFourSummands}):
	\begin{align*}
	D_1^2\CR_2&(F\circ G)(w,x)\\
	= & D_1^w\CR_2(F\circ G)(w,x)\\
	\simeq & D_1^{w}\CR_1F(\CR_2G(w,x))
	\oplus D_1^{w}\CR_2F(\CR_2G(w,x),G(0))\\ 
	&\quad \oplus D_1^{w}\CR_2F(\CR_1G(w),\CR_1G(x)) 
	\oplus D_1^{w}\CR_2F(\CR_2G(w,x),\CR_1G(x))\\
	&\quad \oplus  D_1^{w}\CR_3F(\CR_1G(w),\CR_1G(x),G(0))
	\oplus D_1^{w}\CR_3F(\CR_2G(w,x),\CR_1G(x),G(0)).
	\end{align*}
	
	For the third summand of (\ref{LHSFourSummands}):
	\begin{align*}
	D_1^1D_1^2\CR_2&(F\circ G)(v,\bar{v})\\
	=&D_1^vD_1^{\bar{v}}\CR_2(F\circ G)(v,\bar{v}) \\
	\simeq &D_1^vD_1^{\bar{v}}\CR_1F(\CR_2G(v,\bar{v}))
	\oplus D_1^vD_1^{\bar{v}}\CR_2F(\CR_2G(v,\bar{v}),G(0)) \\
	&\quad \oplus D_1^vD_1^{\bar{v}}\CR_2F(\CR_1G(v),\CR_1G(\bar{v}))
	\oplus D_1^vD_1^{\bar{v}}\CR_3F(\CR_1G(v),\CR_1G(\bar{v}),G(0)).
	\end{align*}
	
	For the fourth summand of (\ref{LHSFourSummands}):
	\begin{align*}
	D_1^2D_1^3&\CR_3(F\circ G)(v,\bar{v},x)\\
	=&D_1^vD_1^{\bar{v}}\CR_3(F\circ G)(v,\bar{v},x)\\
	\simeq & D_1^vD_1^{\bar{v}} \CR_1F(\CR_3G(v,\bar{v},x))\\
	&\oplus D_1^vD_1^{\bar{v}}\CR_2F(\CR_3G(v,\bar{v},x),G(0))
	\oplus D_1^vD_1^{\bar{v}}\CR_2F(\CR_2G(v,\bar{v}),\CR_1G(x)) \\
	&  \oplus  D_1^vD_1^{\bar{v}}\CR_2F(\CR_3G(v,\bar{v},x),\CR_1G(x))  
	\oplus D_1^vD_1^{\bar{v}}\CR_2F(\CR_1G(v),\CR_2G(\bar{v},x))\\
	& \oplus D_1^vD_1^{\bar{v}}\CR_2F(\CR_1G(\bar{v}),\CR_2G(v,x))
	\oplus D_1^vD_1^{\bar{v}}\CR_2F(\CR_2G(v,x),\CR_2G(\bar{v},x)) \\
	& \oplus D_1^vD_1^{\bar{v}}\CR_3F(\CR_2G(v,\bar{v}),\CR_1G(x),G(0)) 
	\oplus D_1^vD_1^{\bar{v}}\CR_3F(\CR_3G(v,\bar{v},x),\CR_1G(x),G(0)) \\
	& \oplus D_1^vD_1^{\bar{v}}\CR_3F(\CR_1G(v),\CR_2G(\bar{v},x),G(0))
	\oplus D_1^vD_1^{\bar{v}}\CR_3F(\CR_1G(\bar{v}),\CR_2G(v,x), G(0))\\
	&\oplus D_1^vD_1^{\bar{v}}\CR_3F(\CR_2G(v,x),\CR_2G(\bar{v},x),G(0))
	\oplus D_1^vD_1^{\bar{v}}\CR_3F(\CR_1G(v),\CR_1G(\bar{v}),\CR_1G(x))\\
	&\oplus D_1^vD_1^{\bar{v}}\CR_3F(\CR_1G(v),\CR_2G(\bar{v},x),\CR_1G(x))
	\oplus D_1^vD_1^{\bar{v}}\CR_3F(\CR_1G(\bar{v}),\CR_2G(v,x),\CR_1G(x))\\
	&\oplus D_1^vD_1^{\bar{v}}\CR_3F(\CR_2G(v,x),\CR_2G(\bar{v},x),\CR_1G(x))\\
	&\oplus D_1^vD_1^{\bar{v}}\CR_4F(\CR_1G(v),\CR_1G(\bar{v}),\CR_1G(x),G(0)) \\
	&\oplus D_1^vD_1^{\bar{v}}\CR_4F(\CR_1G(v),\CR_2G(\bar{v},x),\CR_1G(x),G(0)) \\
	& \oplus D_1^vD_1^{\bar{v}}\CR_4F(\CR_1G(\bar{v}),\CR_2G(v,x),\CR_1G(x),G(0)) \\
	& \oplus D_1^vD_1^{\bar{v}}\CR_4F(\CR_2G(v,x),\CR_2G(\bar{v},x),\CR_1G(x),G(0)).
	\end{align*}
	
	
	All together, the expansion of the left-hand side of Theorem \ref{DirectionalChainRule} is 
	\begin{align}
	&\nonumber \Delta_2(F\circ G)(w,v;x)\simeq\\
	&\label{A1} D_1(F\circ G)(w)\\ 
	&\label{A2} \oplus  D_1^{w}\CR_1F(\CR_2G(w,x))\\
	&\label{A3} \oplus D_1^{w}\CR_2F(\CR_2G(w,x),G(0))\\ 
	&\label{A4} \oplus D_1^{w}\CR_2F(\CR_1G(w),\CR_1G(x)) \\
	&\label{A5} \oplus D_1^{w}\CR_2F(\CR_2G(w,x),\CR_1G(x))\\
	&\oplus\label{A6}  D_1^{w}\CR_3F(\CR_1G(w),\CR_1G(x),G(0))\\
	& \oplus\label{A7} D_1^{w}\CR_3F(\CR_2G(w,x),\CR_1G(x),G(0))\\
	& \label{A8}\oplus D_1^vD_1^{\bar{v}}\CR_1F(\CR_2G(v,\bar{v}))\\ 
	& \label{A9} \oplus D_1^vD_1^{\bar{v}}\CR_2F(\CR_2G(v,\bar{v}),G(0)) \\
	& \label{A10} \oplus D_1^vD_1^{\bar{v}}\CR_2F(\CR_1G(v),\CR_1G(\bar{v}))\\
	&\label{A11} \oplus D_1^vD_1^{\bar{v}}\CR_3F(\CR_1G(v),\CR_1G(\bar{v}),G(0))\\
	&\label{A12} \oplus D_1^vD_1^{\bar{v}} \CR_1F(\CR_3G(v,\bar{v},x))\\
	&\label{A13}\oplus D_1^vD_1^{\bar{v}}\CR_2F(\CR_3G(v,\bar{v},x),G(0))\\
	&\label{A14}\oplus D_1^vD_1^{\bar{v}}\CR_2F(\CR_2G(v,\bar{v}),\CR_1G(x)) \\
	&\label{A15}  \oplus  D_1^vD_1^{\bar{v}}\CR_2F(\CR_3G(v,\bar{v},x),\CR_1G(x))\\
	&\label{A16} \oplus D_1^vD_1^{\bar{v}}\CR_2F(\CR_1G(v),\CR_2G(\bar{v},x))\\
	&\label{A17} \oplus D_1^vD_1^{\bar{v}}\CR_2F(\CR_1G(\bar{v}),\CR_2G(v,x))\\
	&\label{A18} \oplus D_1^vD_1^{\bar{v}}\CR_2F(\CR_2G(v,x),\CR_2G(\bar{v},x)) \\
	&\label{A19}\oplus D_1^vD_1^{\bar{v}}\CR_3F(\CR_2G(v,\bar{v}),\CR_1G(x),G(0))\\ 
	&\label{A20} \oplus D_1^vD_1^{\bar{v}}\CR_3F(\CR_3G(v,\bar{v},x),\CR_1G(x),G(0))\\
	&\label{A21} \oplus D_1^vD_1^{\bar{v}}\CR_3F(\CR_1G(v),\CR_2G(\bar{v},x),G(0))\\
	&\label{A22} \oplus D_1^vD_1^{\bar{v}}\CR_3F(\CR_1G(\bar{v}),\CR_2G(v,x), G(0))\\
	&\label{A23}\oplus D_1^vD_1^{\bar{v}}\CR_3F(\CR_2G(v,x),\CR_2G(\bar{v},x),G(0))\\
	&\label{A24} \oplus D_1^vD_1^{\bar{v}}\CR_3F(\CR_1G(v),\CR_1G(\bar{v}),\CR_1G(x))\\
	&\label{A25}\oplus D_1^vD_1^{\bar{v}}\CR_3F(\CR_1G(v),\CR_2G(\bar{v},x),\CR_1G(x))\\
	&\label{A26}\oplus D_1^vD_1^{\bar{v}}\CR_3F(\CR_1G(\bar{v}),\CR_2G(v,x),\CR_1G(x))\\
	&\label{A27}\oplus D_1^vD_1^{\bar{v}}\CR_3F(\CR_2G(v,x),\CR_2G(\bar{v},x),\CR_1G(x))\\
	&\label{A28}\oplus D_1^vD_1^{\bar{v}}\CR_4F(\CR_1G(v),\CR_1G(\bar{v}),\CR_1G(x),G(0)) \\
	&\label{A29}\oplus D_1^vD_1^{\bar{v}}\CR_4F(\CR_1G(v),\CR_2G(\bar{v},x),\CR_1G(x),G(0)) \\
	&\label{A30} \oplus D_1^vD_1^{\bar{v}}\CR_4F(\CR_1G(\bar{v}),\CR_2G(v,x),\CR_1G(x),G(0)) \\
	&\label{A31} \oplus D_1^vD_1^{\bar{v}}\CR_4F(\CR_2G(v,x),\CR_2G(\bar{v},x),\CR_1G(x),G(0)).
	\end{align}
	
	Notice that each term is labeled individually. In Section \ref{proof}, these terms will be aligned with terms (\ref{B1}) through (\ref{B20}), which come from the right-hand side of Theorem \ref{DirectionalChainRule}. We turn our attention to the expansion of the right-hand side in the next section.
	
	\section{A composition of directional derivatives}\label{RHS}
	
	The right-hand side of Theorem \ref{DirectionalChainRule}, which is a composition of directional derivatives, can also be expanded. Most of the results needed for this expansion were discussed previously in Section \ref{Section:FunctorDirectionalDerivatives}. Specifically, Lemma \ref{nabla1} and Lemma \ref{nabla2} reformulate the expression as a direct sum of linearizations of cross effects, rather than directional derivatives. In addition, Lemma \ref{unreducedG} is necessary to expand the right-hand side in to its smallest component parts in order to align the terms with the left-hand side expansion.
	
	The first step in expanding the right-hand side of Theorem \ref{DirectionalChainRule},
	\[\Delta_2F(\Delta_2G(w,v;x),\Delta_1G(v;x);G(x)),\]
	is to rewrite the second directional derivative of $F$,  $\Delta_2F$, in terms of linearizations of $F$ using Lemma \ref{nabla2}:
	\begin{align}
	\nonumber \Delta_2F&(\Delta_2G(w,v;x),\Delta_1G(v;x);G(x))\\
	&\label{RHSFourSummands}\simeq D_1^1 F(\Delta_2 G(w,v;x))\oplus D_1^1D_1^2\CR_2 F (\Delta_1 G(v;x),\Delta_1 G(v;x))\\
	&\nonumber \quad \oplus D_1^1D_1^2\CR_3 F(\Delta_1 G(v;x),\Delta_1 G(v;x), G(x))\oplus D_1^1\CR_2F(\Delta_2 G(w,v;x),G(x))
	\end{align}

	Notice that $G(x)$ appears as a variable in this expansion. But in the complete expansion of the left-hand side at the end of Section \ref{LHS}, the term $G(x)$ never appears as a variable of one of the cross effects of $F$. Instead it is observed that $\CR_1G(x)$ appears. In order to get a clear correspondence between the expansions of the two sides, a description of the relationship between the occurrence of $G(x)$ versus $\CR_1G(x)$ as a variable of $\CR_k F$ is required.

	\begin{lem}\label{unreducedG}
		Let $G:\cA \rightarrow \cB$ and $F:\cB \rightarrow \cC$ be two composable abelian functors. For $k>0$,
		\begin{align*}
		\CR_k F(x_1,\ldots,x_{k-1},G(x))\cong 
		& \CR_k F(x_1, \ldots, x_{k-1}, G(0))\oplus \CR_k F(x_1,\ldots, x_{k-1},\CR_1 G(x))\\
		&\oplus \CR_{k+1} F(x_1,\ldots,x_{k-1}, G(0),\CR_1 G(x))
		\end{align*}
	\end{lem}
	\begin{proof}
		Using the the definition of the first cross effect, 
		\[G(x)\cong \CR_1G(x)\oplus G(0)\]
		and the definition of the $(k+1)$st cross effect,
		\begin{align*}
		\CR_k F(x_1,\ldots,x_{k-1}, x_k\oplus x_{k+1})\cong 
		& \CR_k F(x_1,\ldots,x_{k-1}, x_k)\oplus \CR_k F(x_1,\ldots,x_{k-1}, x_{k+1})\\
		&\oplus \CR_{k+1} F(x_1,\ldots,x_{k-1}, x_k, x_{k+1}),
		\end{align*}
		it follows that
		\begin{align*}
		\CR_k F(x_1,\ldots,x_{k-1},G(x))
		& \cong \CR_k F(x_1,\ldots,x_{k-1},\CR_1G(x)\oplus G(0))\\
		&\cong \CR_k F(x_1, \ldots, x_{k-1}, \CR_1G(x))\oplus \CR_k F(x_1,\ldots, x_{k-1},G(0))\\
		&\quad\oplus \CR_{k+1} F(x_1,\ldots,x_{k-1},\CR_1 G(x), G(0))
		\end{align*}
	\end{proof}
	
	Applying Lemma \ref{unreducedG} to the third and fourth summands of \eqref{RHSFourSummands}, we obtain
	\begin{align}
	\nonumber&\Delta_2F(\Delta_2G(w,v;x),\Delta_1G(v;x);G(x))\\ 
	&\label{RHS8Summands}\simeq D_1^1 F(\Delta_2 G(w,v;x))
	\oplus D_1^1D_1^2\CR_2 F (\Delta_1 G(v;x),\Delta_1 G(v;x))\\
	&\nonumber \quad \oplus D_1^1D_1^2\CR_3 F(\Delta_1 G(v;x),\Delta_1 G(v;x), G(0))\\
	&\nonumber \quad \oplus D_1^1D_1^2\CR_3 F(\Delta_1 G(v;x),\Delta_1 G(v;x), \CR_1 G(x))\\
	&\nonumber \quad \oplus D_1^1D_1^2\CR_4 F(\Delta_1 G(v;x),\Delta_1 G(v;x),\CR_1 G(x),G(0))
	\oplus D_1^1\CR_2F(\Delta_2 G(w,v;x),G(0))\\
	&\nonumber \quad \oplus D_1^1\CR_2F(\Delta_2 G(w,v;x),\CR_1 G(x))
	\oplus D_1^1\CR_3F(\Delta_2 G(w,v;x), \CR_1 G(x), G(0)).
	\end{align}
	
	We will further expand the right-hand side by working with each of the eight summands of \eqref{RHS8Summands} individually.  Note that by using Lemma \ref{nabla1} and Lemma \ref{nabla2}, we can rewrite $\Delta_1G(v;x)$ and $\Delta_2G(w,v;x)$ respectively as 
	\[\Delta_1G(v;x) \simeq D_1 G(v)\oplus D_1^1\CR_2 G(v,x),\]
	and
	\[\Delta_2 G(w,v;x)\simeq D_1 G(w)\oplus D_1^1D_1^2\CR_2G(v,v)\oplus D_1^1D_1^2\CR_3 G(v,v,x)\oplus D_1^1\CR_2G(w,x).\]
	
	Applying these reformulations of $\Delta_1G(v;x)$ and $\Delta_2G(w,v;x)$ as well as Lemma \ref{lem:D1-linearity}, the first summand of \eqref{RHS8Summands} is
	\begin{align*}
	D_1^1 & F(\Delta_2 G(w,v;x))\\
	&\simeq D_1^1 F\left( D_1 G(w)\oplus D_1^1D_1^2\CR_2G(v,v)\oplus D_1^1D_1^2\CR_3 G(v,v,x)\oplus D_1^1\CR_2G(w,x)\right)\\
	&\simeq D_1^1 F\left( D_1 G(w)\right)
	\oplus D_1^1 F\left(D_1^1D_1^2\CR_2G(v,v)\right)\\
	&\quad \oplus D_1^1 F\left(D_1^1D_1^2\CR_3 G(v,v,x)\right)\oplus D_1^1 F\left(D_1^1\CR_2G(w,x)\right),
	\end{align*}
	the second summand of (\ref{RHS8Summands}) is
	\begin{align*} 
	D_1^1&D_1^2\CR_2 F (\Delta_1 G(v;x),\Delta_1 G(v;x))\\
	&\simeq D_1^1 D_1^2\CR_2 F\left(D_1 G(v)\oplus D_1^1\CR_2 G(v,x),D_1 G(v)\oplus D_1^1\CR_2 G(v,x)\right)\\
	& \simeq D_1^1 D_1^2\CR_2 F\left(D_1 G(v),D_1 G(v)\oplus D_1^1\CR_2 G(v,x)\right)\\
	&\quad \oplus D_1^1 D_1^2\CR_2 F\left(D_1^1\CR_2 G(v,x),D_1 G(v)\oplus D_1^1\CR_2 G(v,x)\right)\\
	& \simeq D_1^1 D_1^2\CR_2 F\left(D_1 G(v), D_1 G(v)\right)\oplus D_1^1 D_1^2\CR_2 F\left(D_1 G(v),D_1^1\CR_2 G(v,x)\right)\\
	&\quad \oplus D_1^1 D_1^2\CR_2 F\left(D_1^1\CR_2 G(v,x),D_1 G(v)\right)\oplus D_1^1 D_1^2\CR_2 F\left(D_1^1\CR_2 G(v,x),D_1^1\CR_2 G(v,x)\right),
	\end{align*}
	the third summand of (\ref{RHS8Summands}) is
	\begin{align*} 
	D_1^1&D_1^2  \CR_3 F\left( \Delta_1 G(v;x),\Delta_1 G(v;x), G(0)\right)\\
	&\simeq D_1^1D_1^2 \CR_3 F\left(D_1 G(v)\oplus D_1^1\CR_2 G(v,x),D_1 G(v)\oplus D_1^1\CR_2 G(v,x), G(0)\right)\\
	&\simeq D_1^1D_1^2 \CR_3 F\left(D_1 G(v),D_1 G(v)\oplus D_1^1\CR_2 G(v,x), G(0)\right)\\
	&\quad \oplus D_1^1D_1^2 \CR_3 F\left(D_1^1\CR_2 G(v,x),D_1 G(v)\oplus D_1^1\CR_2 G(v,x), G(0)\right)\\
	&\simeq D_1^1D_1^2 \CR_3 F\Big(D_1 G(v),D_1 G(v), G(0)\Big) 
	\oplus D_1^1D_1^2 \CR_3 F\left(D_1 G(v), D_1^1\CR_2 G(v,x), G(0)\right)\\
	&\quad \oplus D_1^1D_1^2 \CR_3 F\big(D_1^1\CR_2 G(v,x),D_1 G(v), G(0)\big) \\
	&\quad \oplus D_1^1D_1^2 \CR_3 F\left(D_1^1\CR_2 G(v,x),D_1^1\CR_2 G(v,x), G(0)\right),
	\end{align*}
	the fourth summand of (\ref{RHS8Summands}) is
	\begin{align*} 
	D_1^1&D_1^2 \CR_3 F\left(\Delta_1 G(v;x),\Delta_1 G(v;x), \CR_1 G(x)\right)\\
	&\simeq D_1^1D_1^2 \CR_3 F\left(D_1 G(v)\oplus D_1^1\CR_2 G(v,x),D_1 G(v)\oplus D_1^1\CR_2 G(v,x), \CR_1 G(x)\right)\\
	&\simeq D_1^1D_1^2 \CR_3 F\left(D_1 G(v),D_1 G(v)\oplus D_1^1\CR_2 G(v,x), \CR_1 G(x)\right)\\
	&\quad\oplus D_1^1D_1^2 \CR_3 F\left(D_1^1\CR_2 G(v,x),D_1 G(v)\oplus D_1^1\CR_2 G(v,x), \CR_1 G(x)\right)\\
	&\simeq D_1^1D_1^2 \CR_3 F(D_1 G(v),D_1 G(v), \CR_1 G(x)) 
	\oplus D_1^1D_1^2 \CR_3 F(D_1 G(v), D_1^1\CR_2 G(v,x), \CR_1 G(x))\\
	&\quad \oplus D_1^1D_1^2 \CR_3 F(D_1^1\CR_2 G(v,x),D_1 G(v), \CR_1 G(x))\\
	& \quad \oplus D_1^1D_1^2 \CR_3 F(D_1^1\CR_2 G(v,x),D_1^1\CR_2 G(v,x), \CR_1 G(x)),
	\end{align*}
	the fifth summand of (\ref{RHS8Summands}) is
	\begin{align*} 
	D_1^1&D_1^2 \CR_4 F(\Delta_1 G(v;x),\Delta_1 G(v;x),\CR_1 G(x),G(0))\\
	&\simeq D_1^1D_1^2 \CR_4 F\left(D_1 G(v)\oplus D_1^1\CR_2 G(v,x),D_1 G(v)\oplus D_1^1\CR_2 G(v,x), \CR_1 G(x),G(0)\right)\\
	& \simeq D_1^1D_1^2 \CR_4 F\left(D_1 G(v),D_1 G(v)\oplus D_1^1\CR_2 G(v,x), \CR_1 G(x),G(0)\right)\\
	&\quad\oplus D_1^1D_1^2 \CR_4 F\left( D_1^1\CR_2 G(v,x),D_1 G(v)\oplus D_1^1\CR_2 G(v,x), \CR_1 G(x),G(0)\right)\\
	&\simeq D_1^1D_1^2 \CR_4 F(D_1 G(v),D_1 G(v),\CR_1 G(x),G(0))\\
	& \quad \oplus D_1^1D_1^2 \CR_4 F(D_1 G(v),D_1^1\CR_2 G(v,x),\CR_1 G(x),G(0))\\
	&\quad \oplus D_1^1D_1^2 \CR_4 F( D_1^1\CR_2 G(v,x),D_1 G(v),\CR_1 G(x),G(0))\\
	&\quad \oplus D_1^1D_1^2 \CR_4 F(D_1^1\CR_2 G(v,x),D_1^1\CR_2 G(v,x), \CR_1 G(x),G(0)),
	\end{align*}
	the sixth summand of (\ref{RHS8Summands}) is
	\begin{align*} 
	D_1^1 & \CR_2F(\Delta_2 G(w,v;x),G(0))\\
	\simeq & D_1^1 \CR_2F(D_1 G(w)\oplus D_1^1\CR_2 G(w,x)\oplus D_1^1 D_1^2\CR_2 G(v,v)\oplus D_1^1D_1^2\CR_3 G(v,v,x),G(0))\\
	\simeq & D_1^1 \CR_2F(D_1 G(w),G(0)) \oplus D_1^1 \CR_2F(D_1^1\CR_2 G(w,x),G(0))\\
	& \oplus D_1^1 \CR_2F(D_1^1 D_1^2\CR_2 G(v,v),G(0)) \oplus D_1^1 \CR_2F(D_1^1D_1^2\CR_3 G(v,v,x),G(0)),
	\end{align*}
	the seventh summand of (\ref{RHS8Summands}) is
	\begin{align*} 
	D_1^1& \CR_2F(\Delta_2 G(w,v;x),\CR_1 G(x))\\
	&\simeq D_1^1\CR_2 F(D_1G(w)\oplus D_1^1\CR_2 G(w,x)\oplus D_1^1D_1^2\CR_2 G(v,v)\oplus D_1^1D_1^2\CR_3 G(v,v,x),\CR_1 G(x))\\
	& \simeq  D_1^1 \CR_2F(D_1 G(w),\CR_1 G(x)) \oplus D_1^1 \CR_2F(D_1^1\CR_2 G(w,x),\CR_1 G(x))\\
	&\quad \oplus D_1^1 \CR_2F(D_1^1 D_1^2\CR_2 G(v,v),\CR_1 G(x)) \oplus D_1^1 \CR_2F(D_1^1D_1^2\CR_3 G(v,v,x),\CR_1 G(x)),
	\end{align*}
	and the eighth summand of (\ref{RHS8Summands}) is
	\begin{align*} 
	D_1^1& \CR_3F(\Delta_2 G(w,v;x),\CR_1 G(x),G(0))\\
	&\simeq D_1^1 \CR_3F(D_1G(w)\oplus D_1^1\CR_2 G(w,x)\oplus D_1^1D_1^2\CR_2 G(v,v)\\
	&\quad\oplus D_1^1D_1^2\CR_3 G(v,v,x), \CR_1 G(x),G(0))\\
	&\simeq D_1^1 \CR_3F(D_1G(w), \CR_1 G(x), G(0))\oplus D_1^1 \CR_3F(D_1^1\CR_2 G(w,x), \CR_1 G(x), G(0))\\
	&\quad \oplus D_1^1 \CR_3F(D_1^1D_1^2\CR_2 G(v,v), \CR_1 G(x),G(0))\\
	&\quad \oplus D_1^1 \CR_3F(D_1^1D_1^2\CR_3 G(v,v,x), \CR_1 G(x), G(0)).
	\end{align*}
	
	
	Putting all of these expansions together, the  expansion for the right-hand side is
	\begin{align}
	\nonumber \Delta_2F&(\Delta_2G(w,v;x),\Delta_1G(v;x);G(x)) \simeq  \\ 
	\label{B1} & D_1F( D_1 G(w))\\  
	& \label{B8} \oplus D_1 F(D_1^1D_1^2\CR_2G(v,v))\\
	&\label{B12} \oplus D_1 F(D_1^1D_1^2\CR_3 G(v,v,x))\\ 
	& \label{B2} \oplus D_1 F(D_1^1\CR_2G(w,x))\\
	&\label{B10} \oplus D_1^1 D_1^2\CR_2 F\left(D_1 G(v), D_1 G(v)\right)\\ 
	&\label{B16} \oplus D_1^1 D_1^2\CR_2 F\left(D_1 G(v),D_1^1\CR_2 G(v,x)\right)\\
	&\label{B17} \oplus D_1^1 D_1^2\CR_2 F\left(D_1^1\CR_2 G(v,x),D_1 G(v)\right)\\
	&\label{B18} \oplus D_1^1 D_1^2\CR_2 F\left(D_1^1\CR_2 G(v,x),D_1^1\CR_2 G(v,x)\right)\\
	&\label{B11} \oplus D_1^1D_1^2 \CR_3 F(D_1 G(v),D_1 G(v), G(0))\\ 
	&\label{B21} \oplus D_1^1D_1^2 \CR_3 F(D_1 G(v), D_1^1\CR_2 G(v,x), G(0))\\
	&\label{B22}  \oplus D_1^1D_1^2 \CR_3 F(D_1^1\CR_2 G(v,x),D_1 G(v), G(0))\\
	&\label{B23} \oplus D_1^1D_1^2 \CR_3 F(D_1^1\CR_2 G(v,x),D_1^1\CR_2 G(v,x), G(0))\\
	&\label{B24} \oplus D_1^1D_1^2 \CR_3 F(D_1 G(v),D_1 G(v), \CR_1 G(x))\\ 
	&\label{B25} \oplus D_1^1D_1^2 \CR_3 F(D_1 G(v), D_1^1\CR_2 G(v,x), \CR_1 G(x))\\
	&\label{B26}  \oplus D_1^1D_1^2 \CR_3 F(D_1^1\CR_2 G(v,x),D_1 G(v), \CR_1 G(x))\\
	&\label{B27} \oplus D_1^1D_1^2 \CR_3 F(D_1^1\CR_2 G(v,x),D_1^1\CR_2 G(v,x), \CR_1 G(x))\\
	&\label{B28} \oplus D_1^1D_1^2 \CR_4 F(D_1 G(v),D_1 G(v),\CR_1 G(x),G(0))\\
	&\label{B29} \oplus D_1^1D_1^2 \CR_4 F(D_1 G(v),D_1^1\CR_2 G(v,x),\CR_1 G(x),G(0))\\
	&\label{B30} \oplus D_1^1D_1^2 \CR_4 F( D_1^1\CR_2 G(v,x),D_1 G(v),\CR_1 G(x),G(0))\\
	&\label{B31} \oplus D_1^1D_1^2 \CR_4 F(D_1^1\CR_2 G(v,x),D_1^1\CR_2 G(v,x), \CR_1 G(x),G(0))\\
	&\label{C1} \oplus D_1^1 \CR_2F(D_1 G(w),G(0))\\ 
	&\label{B3} \oplus D_1^1 \CR_2F(D_1^1\CR_2 G(w,x),G(0))\\
	&\label{B9} \oplus D_1^1 \CR_2F(D_1^1 D_1^2\CR_2 G(v,v),G(0))\\ 
	&\label{B13} \oplus D_1^1 \CR_2F(D_1^1D_1^2\CR_3 G(v,v,x),G(0))\\
	&\label{B4} \oplus D_1^1 \CR_2F(D_1 G(w),\CR_1 G(x))\\ 
	&\label{B5} \oplus D_1^1 \CR_2F(D_1^1\CR_2 G(w,x),\CR_1 G(x))\\
	&\label{B14} \oplus D_1^1 \CR_2F(D_1^1 D_1^2\CR_2 G(v,v),\CR_1 G(x))\\
	&\label{B15} \oplus D_1^1 \CR_2F(D_1^1D_1^2\CR_3 G(v,v,x),\CR_1 G(x))\\
	&\label{B6} \oplus D_1^1 \CR_3F(D_1G(w), \CR_1 G(x), G(0))\\ 
	& \label{B7} \oplus D_1^1 \CR_3F(D_1^1\CR_2 G(w,x), \CR_1 G(x), G(0))\\
	&\label{B19} \oplus D_1^1 \CR_3F(D_1^1D_1^2\CR_2 G(v,v), \CR_1 G(x), G(0))\\
	&\label{B20} \oplus D_1^1 \CR_3F(D_1^1D_1^2\CR_3 G(v,v,x), \CR_1 G(x), G(0)).
	\end{align}
	
	
	\section{Proof of the chain rule for the second directional derivative}\label{proof}
	
	All of the key pieces to prove Theorem \ref{DirectionalChainRule} are built.  Specifically, we have expanded the left-hand side of Theorem \ref{DirectionalChainRule} in Section \ref{LHS}, and we have expanded the right-hand side in Section \ref{RHS}.  
	
	In the proof, we will use two cases of the chain rule for abelian funtors. First, there is a chain rule for $D_1$ if the interior functor is reduced.
	\begin{lem}\label{reducedChainRule}\cite[Proposition 5.7]{BJORT}
		If $G:\cA\rightarrow \cB$ and $F:\cB \rightarrow \cC$  are composable abelian functors and $G$ is a reduced functor, then there is a chain homotopy equivalence 
		\[D_1(F\circ G)(x)\simeq D_1 F\circ D_1 G(x).\]
	\end{lem}
	There is also a chain rule for $D_1$ if the interior functor is not reduced, but an additional correction term is required.
	\begin{lem}\label{D1ChainRule} \cite[Proposition 5.10]{BJORT}
		If $G:\cA\rightarrow \cB$ and $F:\cB \rightarrow \cC$  are composable abelian functors, then there is a chain homotopy equivalence
		\[D_1(F\circ G)(x)\simeq D_1F\circ D_1G(x)\oplus D_1^x\CR_2F(\CR_1G(x),G(0)).\]
	\end{lem}
	
	We make a few observations concerning $D_1$.
	\begin{obs}\label{D1CR1Trick}
		Let $F$ be an abelian functor. The linearization of $\CR_1F$ is chain homotopic to the linearization of $F$. In other words,
		\[D_1\CR_1 F(x)\simeq D_1 F(x).\]
	\end{obs}
	\begin{proof}
		
		Recall that $\CR_1 F(0)\cong 0$ because cross effects are multi-reduced. 
		In order to compute $\CR_1(\CR_1 F(-))(x)$, we consider the definition of the first cross effect of the functor $cr_1F$:
		\[\CR_1(\CR_1 F)(x)\oplus\CR_1 F(0)\cong\CR_1 F(x),\]
		and thus $\CR_1\CR_1 F(x)\cong \CR_1 F(x)$. This further implies that $\CR_2(\CR_1 F(-))(x,y)\cong\CR_2 F(x,y)$. 
		
		Recalling the definition of the linearization of $F$, observe that when we linearize $F(x)$ and $\CR_1 F(x)$, we construct equivalent complexes. 
	\end{proof}

	\begin{obs}\label{ObsLinearIsLinear}
		Let $F$ be a functor between abelian categories. The linearization of $D_1F$ is chain homotopy equivalent to the linearization of $F$. In other words,
		\[ D_1 D_1 F(x)\simeq D_1 F(x).\]
	\end{obs}
	\begin{proof}
		Recall that $D_1 F$ is reduced and degree 1. It follows that $\CR_1 D_1 F(x) \simeq D_1 F(x)$ and $\CR_2 D_1 F(x,y) \simeq 0$. If we linearize $D_1 F$, then we have
		\[ (D_1 D_1F)_k \coloneqq \begin{cases} 
		0& k\geq 1\\
		D_1 F & k=0 \\
		0 & \text{otherwise,}\end{cases}\]
		which is equivalent to $D_1 F$. 
	\end{proof}
	
	Now, we proceed with the proof of the main theorem.	
	\renewcommand{\thethm}{\ref{DirectionalChainRule}}
	\begin{thm}
		Given two composable abelian functors $G:\cA \rightarrow \cB$ and $F:\cB \rightarrow \cC$ with object $x$, $v$, and $w$ in $\cA$,  there is a chain homotopy equivalence  
		\begin{equation*}\Delta_2(F\circ G)(w,v;x)\simeq \Delta_2F(\Delta_2G(w,v;x),\Delta_1G(v;x);G(x)).\end{equation*}
	\end{thm}
	\addtocounter{thm}{-1}
	
	\begin{proof}
		
		We will show homotopy equivalence by matching the summands on the left-hand side (terms \eqref{A1} through \eqref{A31}) with their homotopy equivalents on the right-hand side (terms \eqref{B1} through \eqref{B20}). The justifications for equivalence between these terms are very similar. With this in mind, we will prove just one case of each type and list the remainder of the pairs of terms.

		\emph{Type 1:} $\eqref{A3} \simeq \eqref{B3}$. Only Lemma \ref{reducedChainRule} is needed. 
		\begin{align*}
		\eqref{A3} &=D_1^{w}\CR_2F(\CR_2G(w,x),G(0))\\
		&=D_1(\CR_2F(-,G(0))\circ \CR_2G(-,x))(w)\\
		&\simeq D_1^1\CR_2F(D_1^1\CR_2G(w,x),G(0))\\
		&=\eqref{B3}.
		\end{align*}
		The proofs of $\eqref{A5}\simeq \eqref{B5}$ and $\eqref{A7}\simeq \eqref{B7}$ are similar. 
		
		\emph{Type 2:} $\eqref{A2}\simeq \eqref{B2}$. Lemma \ref{reducedChainRule} is used, followed by Observation \ref{D1CR1Trick}.
		\begin{align*}
		\eqref{A2}&=D_1^w\CR_1F(\CR_2G(w,x))\\
		&=D_1(\CR_1F(-)\circ\CR_2G(-,x))(w)\\
		&\simeq D_1\CR_1F(-)\circ D_1\CR_2G(-,x)(w)\\
		&=D_1\CR_1F(D_1^1\CR_2G(w,x))\\
		&\simeq D_1F(D_1^1\CR_2G(w,x))\\
		&=\eqref{B2}.
		\end{align*}
		The proofs of $\eqref{A4} \simeq \eqref{B4}$ and $\eqref{A6}\simeq \eqref{B6}$ are similar.
		
		\emph{Type 3:} $\eqref{A1}\simeq \eqref{B1}\oplus\eqref{C1}$. Lemma \ref{D1ChainRule} is applied, followed by Lemma \ref{reducedChainRule} and Observation \ref{D1CR1Trick}.
		\begin{align*}
		\eqref{A1}&=D_1(F\circ G)(w)\\
		&\simeq D_1\circ D_1G(w)\oplus D_1^w\CR_2F(\CR_1G(w),G(0))\\
		&=D_1F(D_1G(w))\oplus D_1(\CR_2F(-,G(0))\circ \CR_1G(-))(w)\\
		&\simeq D_1F(D_1G(w)) \oplus D_1^1\CR_2F(D_1\CR_1G(w),G(0))\\
		&\simeq D_1F(D_1G(w)) \oplus D_1^2\CR_2F(D_1G(w), G(0))\\
		&=\eqref{B1}\oplus \eqref{C1}.
		\end{align*}
		
		\emph{Type 4:} $\eqref{A9}\simeq \eqref{B9}$. Lemma \ref{D1ChainRule} is applied twice, as well as Observation \ref{ObsLinearIsLinear}: 
		\begin{align*}
		\eqref{A9}&=D_1^vD_1^{\bar{v}}\CR_2F(\CR_2G(v,\bar{v}),G(0))\\
		&=D_1^v[D_1(\CR_2F(-,G(0))\circ \CR_2G(v,-))(\bar{v})]\\
		&\simeq D_1^v[D_1^1\CR_2F(D_1^2\CR_2G(v,\bar{v}),G(0))]\\
		&=D_1[D_1^1\CR_2F(-,G(0))\circ D_1^2\CR_2G(-,\bar{v})](v)\\
		&\simeq D_1^1D_1^1\CR_2F(D_1^1D_1^2\CR_2G(v,\bar{v}),G(0))\\
		&\simeq D_1^1\CR_2F(D_1^1D_1^2\CR_2G(v,\bar{v}),G(0))\\
		&=\eqref{B9}.
		\end{align*}
		The proofs of $\eqref{A13}\simeq \eqref{B13}$, $\eqref{A14}\simeq \eqref{B14}$, $\eqref{A15}\simeq\eqref{B15}$, $\eqref{A18}\simeq \eqref{B18}$, $\eqref{A19}\simeq \eqref{B19}$, $\eqref{A20}\simeq \eqref{B20}$, $\eqref{A23}\simeq \eqref{B23}$, $\eqref{A27}\simeq \eqref{B27}$, and $\eqref{A31}\simeq \eqref{B31}$ are similar. 
		
		\emph{Type 5:}  $\eqref{A8}\simeq \eqref{B8}$.  Lemma \ref{D1ChainRule} is applied twice, as well as Observations \ref{ObsLinearIsLinear} and \ref{D1CR1Trick}: 
		\begin{align*}
		\eqref{A8}&=D_1^vD_1^{\bar{v}}\CR_1F(\CR_2G(v,\bar{v})\\
		&=D_1^v[D_1(\CR_1F\circ \CR_2G(v,-))(\bar{v})]\\
		&\simeq D_1^v[D_1\CR_1F(D_1^2\CR_2G(v,\bar{v}))]\\
		&=D_1[D_1\CR_1F\circ D_1^2\CR_2G(-,\bar{v})](v)\\
		&\simeq D_1D_1\CR_1F(D_1^1D_1^2\CR_2G(v,\bar{v}))\\
		&\simeq D_1\CR_1F(D_1^1D_1^2\CR_2G(v,\bar{v}))\\
		&\simeq D_1F(D_1^1D_1^2\CR_2G(v,\bar{v}))\\
		&=\eqref{B8}.
		\end{align*}
		The proofs of $\eqref{A10}\simeq \eqref{B10}$, $\eqref{A11}\simeq \eqref{B11}$, $\eqref{A12}\simeq\eqref{B12}$, $\eqref{A16}\simeq \eqref{B16}$, $\eqref{A17}\simeq \eqref{B17}$, $\eqref{A21}\simeq \eqref{B21}$, $\eqref{A22}\simeq\eqref{B22}$, $\eqref{A24}\simeq\eqref{B24}$, $\eqref{A25}\simeq\eqref{B25}$, $\eqref{A26}\simeq\eqref{B26}$, $\eqref{A28}\simeq \eqref{B28}$, $\eqref{A29}\simeq \eqref{B29}$, and $\eqref{A30}\simeq \eqref{B30}$ are similar. 
		
	\end{proof}
	
	\section{Conclusion}
	
	We proved the chain rule formula for the second higher order directional derivative using primarily properties of linearization and cross effects. This result gave the authors of \cite{BJORT} hope that their definition of higher order directional derivatives of functors would produce a higher order directional derivative chain rule,
	\[\Delta_n(F\circ G)\left(v_n,\ldots,v_1;x_0\right)\simeq\Delta_nF\left( \Delta_nG\left(v_n,\ldots,v_1;x_0\right),\ldots, \Delta_1G\left(v_1;x_0\right);G\left(x_0\right)\right),\]
	which mirrors the analogous result for functions (see \cite[Theorem 3]{HMY}).
	The proof strategy used in this paper does not provide a clear inductive procedure that could lead to the more general result. 
	Thus more sophisticated machinery was developed to prove the higher order directional derivative chain rule for abelian functors \cite[Theorem 8.1]{BJORT}.

\end{document}